\documentclass{article}

\usepackage{graphicx}
\usepackage{color}
\usepackage{indentfirst}
\usepackage{amsmath,amsfonts,amsthm,amssymb}
\usepackage{mathrsfs}
\usepackage{amscd}
\usepackage{hyperref}

\def\co{\colon\thinspace}
\DeclareMathAlphabet{\mathsfsl}{OT1}{cmss}{m}{sl}

\newcommand{\spin}{\mathrm{Spin}^c}
\newcommand{\HFtildered}{\widetilde{HF}_{\mathrm{red}}}

\newtheorem{thm}{Theorem}[section]
\newtheorem{lem}[thm]{Lemma}
\newtheorem{cor}[thm]{Corollary}
\newtheorem{prop}[thm]{Proposition}

\newtheorem*{thm*}{Theorem}

\theoremstyle{definition}

\newtheorem{rem}[thm]{Remark}

\newtheorem{conv}[thm]{Convention}
\newtheorem{exam}[thm]{Example}
\newtheorem{ques}[thm]{Question}

\begin{document}

\def\G{{\Gamma}}
  \def\d{{\delta}}
  \def\ci{{\circ}}
  \def\e{{\epsilon}}
  \def\l{{\lambda}}
  \def\L{{\Lambda}}
  \def\m{{\mu}}
  \def\n{{\nu}}
  \def\o{{\omega}}
  \def\O{{\Omega}}
  \def\Th{{\Theta}}\def\s{{\sigma}}
  \def\v{{\varphi}}
  \def\a{{\alpha}}
  \def\b{{\beta}}
  \def\p{{\partial}}
  \def\r{{\rho}}
  \def\ra{{\rightarrow}}
  \def\lra{{\longrightarrow}}
  \def\g{{\gamma}}
  \def\D{{\Delta}}
  \def\La{{\Leftarrow}}
  \def\Ra{{\Rightarrow}}
  \def\x{{\xi}}
  \def\c{{\mathbb C}}
  \def\z{{\mathbb Z}}
  \def\2{{\mathbb Z_2}}
  \def\q{{\mathbb Q}}
  \def\t{{\tau}}
  \def\u{{\upsilon}}
  \def\th{{\theta}}
  \def\la{{\leftarrow}}
  \def\lla{{\longleftarrow}}
  \def\da{{\downarrow}}
  \def\ua{{\uparrow}}
  \def\nwa{{\nwtarrow}}
  \def\swa{{\swarrow}}
  \def\nea{{\netarrow}}
  \def\sea{{\searrow}}
  \def\hla{{\hookleftarrow}}
  \def\hra{{\hookrightarrow}}
  \def\sl{{SL(2,\mathbb C)}}
  \def\ps{{PSL(2,\mathbb C)}}
  \def\qed{{\hspace{2mm}{\small $\diamondsuit$}\goodbreak}}
  \def\pf{{\noindent{\bf Proof.\hspace{2mm}}}}
  \def\ni{{\noindent}}
  \def\sm{{{\mbox{\tiny M}}}}
   \def\sf{{{\mbox{\tiny F}}}}
   \def\sc{{{\mbox{\tiny C}}}}
  \def\ke{{\mbox{ker}(H_1(\partial M;\2)\ra H_1(M;\2))}}
  \def\et{{\mbox{\hspace{1.5mm}}}}

\title{Characterizing slopes for torus knots}

\author{{Yi NI}\\{\normalsize Department of Mathematics, Caltech}\\
{\normalsize 1200 E California Blvd, Pasadena, CA 91125}
\\{\small\it Email\/:\quad\rm yini@caltech.edu}
\\\\
{Xingru ZHANG}
\\
{\normalsize Department of Mathematics,
University at Buffalo}\\
{\small\it Email\/:\quad\rm xinzhang@buffalo.edu}}

\date{}
\maketitle

\begin{abstract}A   slope $\frac pq$ is called a characterizing slope for a given knot $K_0$ in $S^3$
if whenever the $\frac pq$-surgery  on a knot $K$ in $S^3$ is homeomorphic to the $\frac pq$-surgery on $K_0$ via an orientation preserving homeomorphism, then $K=K_0$.
In this paper we try to find characterizing slopes for torus knots $T_{r,s}$.
We show that  any slope
$\frac pq$ which is larger than the number $\frac{30(r^2-1)(s^2-1)}{67}$ is a characterizing slope
for $T_{r,s}$.
 The proof uses Heegaard Floer homology and Agol--Lackenby's 6--Theorem. In the case of $T_{5,2}$, we
 obtain more specific information about its set of characterizing slopes by applying more Heegaard Floer homology techniques.
\end{abstract}

\section{Introduction}\label{Sect:Intr}

There had been a long standing conjecture due to Gordon that if for some nontrivial slope $\frac pq\ne \frac 10$, the $\frac pq$-surgery on a knot $K\subset S^3$ is homeomorphic to the $\frac pq$-surgery on the unknot in $S^3$ via an orientation preserving homeomorphism, then $K$ must be the unknot.
This conjecture was originally proved using Monopole Floer homology \cite{KMOSz}, and there were also proofs via Heegaard Floer homology \cite{OSzGenus,OSzRatSurg}.
 It is natural  to ask if there are other knots in $S^3$ which admit a similar Dehn surgery characterization as the unknot. In \cite{OSz3141} Ozsv\'ath and Szab\'o  showed  that the trefoil knot and the figure--$8$ knot are two such knots.
In the proof of these results, one uses the fact that the unknot is the the only genus zero knot, and the trefoil knot and the figure--$8$ knot are the only genus one fibred knots.

For a given knot $K_0\subset S^3$, we call a  slope $\frac pq$ {\it a characterizing slope for $K_0$} if whenever the $\frac pq$-surgery  on a knot $K$ in $S^3$ is homeomorphic to the $\frac pq$-surgery on $K_0$ via an orientation preserving homeomorphism, then $K=K_0$.
In this terminology, the results cited above were saying that for each of the unknot, the trefoil knot
and the figure--$8$ knot, every nontrivial slope is a characterizing slope.

On the other hand there are  infinitely many knots in $S^3$  which have nontrivial  non-characterizing slopes, including some genus two fibred knots.
Osoinach \cite{Os} found examples of infinitely many knots in $S^3$ on which the $0$-surgery yields the same manifold. In one case, the knots he constructed are all genus 2 fibred knots, one of which is the connected sum of two copies of the figure--$8$ knot.
Teragaito \cite{Tera} constructed infinitely many knots on which the $4$-surgery yields the same Seifert fibred space over the base orbifold $S^2(2,6,7)$. One of these knots is $9_{42}$, again a genus $2$ fibred knot.
The following two examples show that some torus knots also have nontrivial non-characterizing slopes.

 \begin{exam}
 Consider the $21$--surgery on the torus knots $T_{5,4}$ and $T_{11,2}$. By \cite{Moser}, the resulting oriented manifolds are the lens spaces $L(21,16)$ and $L(21,4)$. Here our orientation on the lens space $L(p,q)$ is induced from the orientation on $S^3$ by $\frac pq$--surgery on the unknot. Since $16\cdot 4=64\equiv1\pmod{21}$, $L(21,16)\cong L(21,4)$, where (and throughout the paper) $\cong$ standards for orientation preserving homeomorphism. Similarly, the $(n^3+6n^2+10n+4)$--surgery on $T_{n^2+3n+1,n+3}$ and $T_{n^2+5n+5,n+1}$ give rise to homeomorphic oriented lens spaces.
\end{exam}

\begin{exam}
Let $K$ be the $(59,2)$--cable of $T_{6,5}$, denoted $K=C_{59,2}\circ T_{6,5}$, and consider the $119$--surgery on $T_{24,5}$ and $K$. The resulting manifold of the surgery on $K$ is the same as the $\frac{119}4$ surgery on $T_{6,5}$, which is the lens space $L(119,100)$. The $119$--surgery on $T_{24,5}$ is the lens space $L(119,25)$. Since $25\cdot 100\equiv1\pmod{119}$, $L(119,100)\cong L(119,25)$. There are infinitely many such pairs of torus knot and cable knot, of which the first few are listed below:
\begin{center}
 \begin{tabular}{|c|c|c|c|c|}
 \hline
$p$ & 119 &697 &4059 &23661 \\ \hline
knots &
$T_{24,5}, C_{59,2}\circ T_{6,5}$ &$T_ {29,24}, C_{349,2}\circ T_{29,6}$ &$T_{140,29}, C_{2029,2}\circ T_{35,29}$ &$T_{169,140}, C_{11831,2}\circ T_{169,35}$ \\
\hline
 \end{tabular}
\end{center}\end{exam}

In general it is a challenging  problem to determine,  for a given knot, its set of characterizing
slopes.
 Our aim in this paper is to focus on torus knots, trying to find their sets of characterizing slopes.
 In this regard, Rasmussen \cite{RasBerge} proved, using a result of Baker \cite{Baker}, that  $(4n+3)$ is a characterizing slope for  $T_{2n+1,2}$.
It was also known, as a consequence of a result of Greene \cite{GreeneRed},  that $rs$ is a characterizing slope for  $T_{r,s}$.
Our first result is the following theorem.

\begin{thm}\label{thm:TorusKnot}
For a torus knot $T_{r,s}$ with  $r>s>1$,  a nontrivial slope $\frac pq$ is a characterizing slope
if it is  larger than the number $\frac{30(r^2-1)(s^2-1)}{67}$.
\end{thm}

The proof of Theorem~\ref{thm:TorusKnot} uses a genus bound on $K$ from Heegaard Floer homology, as well as the 6--Theorem of Agol \cite{Agol} and Lackenby \cite{Lack}.
A more careful study involving Heegaard Floer homology should definitely give a smaller bound, which is expected to be of order $rs$.
For example, in the special case of $T_{5,2}$, we have the following theorem.

\begin{thm}\label{thm:T52}
For the torus knot $T_{5,2}$, its set of characterizing slopes includes
the set of slopes
\begin{center}$\displaystyle\left\{\frac pq>1, |p|\ge 33\right\}\cup\left\{
\frac pq<-6, |p|\ge33, |q|\ge2\right\}\cup \left\{\frac pq,  |q|\ge9\right\}
\cup\left\{\frac pq,   |q|\geq 3, 2\leq |p|\leq 2|q|-3\right\}$\\
$\displaystyle\cup\left\{9,10,11,\frac{19}{2},\frac{21}2,
\frac{28}{3}, \frac{29}{3}, \frac{31}{3}, \frac{32}{3} \right\}.$
\end{center}
\end{thm}

Theorem \ref{thm:T52}
is essentially saying that besides some  finite set of slopes $p/q$ with $-47\leq p\leq 32$ and $1\le q\leq 8$,
only negative integer slopes could possibly be  nontrivial non-characterizing slopes for $T_{5,2}$.
We suspect that for $T_{5,2}$ every nontrivial slope is a characterizing slope.

Throughout this paper our notations are consistent. A slope $\frac pq$ of  a knot in $S^3$ is always parameterized with respect to the standard meridian-longitude coordinates
 so that a meridian has slope $\frac 10$ and a longitude $0$, and  $p, q$ are always assumed to be relatively prime.
Our definition for the $(m,n)$-cabled knot $K$ on a knot  $K'$ in $S^3$ is standard,  i.e.
 $m, n$ are a pair of  relatively prime integers with $|n|>1$ and $K$
  can be embedded in the boundary torus of a regular neighborhood of $K'$
having slope $\frac mn$ for the knot $K'$.
In particular the $(m,n)$-cabled knot on the unknot is called a torus knot, which we denote by $T_{m,n}$.
Given a knot $K\subset S^3$,
   $S^3_{p/q}(K)$ will denote the oriented $3$-manifold obtained by $\frac pq$-surgery on $S^3$ along  $K$,
 with the orientation induced from that of $S^3-K$ whose orientation is in turn induced from
 a fixed orientation of $S^3$. Similarly if $L$ is a null-homologous knot in a rational
  homology sphere $Y$, $Y_{p/q}(L)$ will denote the oriented $3$-manifold obtained by
  the $\frac pq$-surgery on $Y$ along $L$, with $\frac pq$ parameterized by the standard meridian-longitude
  coordinates of $L$.
  For a knot $K$ in $S^3$, $g(K)$ will denote the genus of $K$.
For two slopes $\frac pq$ and $\frac mn$ of a knot, $\D(\frac pq, \frac mn)$ will denote the ``distance''
between the two slopes, which is equal to $|pn-qm|$.

We conclude this section by raising  the following two questions for which  solutions might not be too remote
to reach.

 \begin{ques}Can every nontrivial slope be realized as a non-characterizing slope
 for some knot in $S^3$?
 \end{ques}

\begin{ques}If $K_0$ is a hyperbolic knot and $\frac pq$ is a  slope with $|p|+|q|$ sufficiently large,
 must $\frac pq$ be a characterizing slope for $K_0$?
 \end{ques}

The paper is organized as follows. In Section~\ref{sect:TorusKnot}, we prove Theorem~\ref{thm:TorusKnot} by using a genus bound from Heegaard Floer homology and Agol--Lackenby's 6--Theorem. The rest of the paper is dedicated to the proof of Theorem~\ref{thm:T52}. In Section~\ref{sect:Pre}, we give some necessary background on Heegaard Floer homology, including Ozsv\'ath--Szab\'o's rational surgery formula. We prove an explicit formula for $HF_{\mathrm{red}}(S^3_{p/q}(K))$ (Proposition~\ref{prop:CompRed}), which is of independent interest.  In Section~\ref{sect:GenusBound}, we deduce information about knot Floer homology under the
 Dehn surgery condition: $S^3_{p/q}(K)\cong S^3_{p/q}(T_{5,2})$ for $\frac pq>1$, or $\frac pq<-6$ and $|q|>2$, and conclude that $K$ is  either a genus $2$ fibred knot or a genus $1$ knot. In Section~\ref{sect:T52}, we finish the proof of Theorem~\ref{thm:T52} by applying the results obtained in  the early sections. The proof also
 applies a result of Lackenby and Meyerhoff \cite{LM}.

\vspace{5pt}\noindent{\bf Acknowledgements.}\quad The first author was
partially supported by an AIM Five-Year Fellowship, NSF grant
number DMS-1103976 and an Alfred P. Sloan Research Fellowship.  We are grateful to the organizers of the ``Workshop on Topics in Dehn Surgery''
at University of Texas at Austin for making our collaboration possible. We wish to thank John Baldwin for a conversation which made us realize a mistake in an earlier version of this paper.

\section{Proof of Theorem~\ref{thm:TorusKnot}}\label{sect:TorusKnot}

Throughout this section, $T_{r,s}$ is the torus knot given in Theorem~\ref{thm:TorusKnot} with $r>s>1$.
In order to prove Theorem~\ref{thm:TorusKnot}, we need a little bit knowledge on Heegaard Floer homology.
Recall that a rational homology sphere $Y$ is an {\it L-space} if the rank of its Heegaard Floer homology $\widehat{HF}(Y)$ is equal to the order of $H_1(Y;\mathbb Z)$. For example, lens spaces are L-spaces. If a knot $K\subset S^3$ admits an L-space surgery with positive slope, then $S^3_{p/q}(K)$ is an L-space if and only if $\frac pq\ge2g(K)-1$ \cite{OSzRatSurg}.

\begin{prop}\label{prop:GenusBound}
Suppose that for a knot $K\subset S^3$ and  a slope $\frac pq\ge rs-r-s$,
 we have $$S^3_{p/q}(K)\cong S^3_{p/q}(T_{r,s}).$$
Then $K$ is a fibred knot and $$2g(K)-1\le\frac{(r^2-1)(s^2-1)}{24}.$$
\end{prop}
\begin{proof}
Since $\frac pq\ge rs-r-s=2g(T_{r,s})-1$, $S^3_{p/q}(T_{r,s})$ is an L-space. As $K$ admits an L-space surgery, according to \cite{OSzLens} its Alexander polynomial should be
$$\Delta_K(T)=(-1)^k+\sum_{i=1}^k(-1)^{k-i}(T^{n_i}+T^{-n_i})$$
for some positive integers $0<n_1<n_2<\cdots<n_k=g(K)$. Moreover, $K$ is fibred \cite{NiFibred}. We can compute
$$\Delta''_K(1)=2\sum_{i=1}^k(-1)^{k-i}n_i^2\ge2(2n_k-1)=2(2g(K)-1).$$

Let $\lambda(Y)$ be the Casson--Walker invariant \cite{Walker} for a rational homology sphere $Y$, normalized so that $\lambda(S^3_{+1}(T_{3,2}))=1$. The surgery formula
$$\lambda(Y_{p/q}(L))=\lambda(Y)+\lambda(L(p,q))+\frac{q}{2p}\Delta_L''(1)$$
is well-known, where $L\subset Y$ is a null-homologous knot. Applying the formula to our case, we conclude that
\begin{equation}\label{eq:Delta''}
\Delta''_{K}(1)=\Delta''_{T_{r,s}}(1),
\end{equation}
hence $2(2g(K)-1)\le \Delta''_{T_{r,s}}(1)=\frac{(r^2-1)(s^2-1)}{12}$. We get our conclusion.
\end{proof}

\begin{lem}\label{lem:pBound}
Suppose that $K$ is a hyperbolic knot in $S^3$ and $S^3_{p/q}(K)$ is not a hyperbolic manifold, then $$|p|\le\frac{36}{3.35}(2g(K)-1)<10.75(2g(K)-1).$$
\end{lem}
\begin{proof}
Following \cite{Agol}, let $C$ be the maximal horocusp of $S^3-K$ with embedded interior. Note that
  $\partial C$ is a Euclidean torus. For any slope $\alpha$ on $\partial C$, let $l_C(\alpha)$ be the Euclidean length of the geodesic loop in the homology class of $\alpha$ in $\p C$. Let $\lambda\subset \p C$ be the canonical longitude of $K$, then it follows from \cite[Theorem~5.1]{Agol} that $$l_C(\lambda)\le 6(2g(K)-1).$$

Since  $S^3_{p/q}(K)$ is not hyperbolic, the 6--Theorem \cite{Lack,Agol} implies that $l_C(\frac pq)\le 6$ (Note that by the geometrization theorem due to Perelman \cite{P1, P2}, for closed $3$-manifolds, the term  hyperbolike, as used in \cite{Agol}, is equivalent to hyperbolic).
Let $\theta$ be the angle between the two geodesic loops in the homology classes $\lambda,\frac pq$.
As in the proof of \cite[Theorem~8.1]{Agol}
$$|p|=\Delta(\lambda,\frac pq)=\frac{l_C(\lambda)l_C(\frac pq)\sin\theta}{\mathrm{Area}(\partial C)}\le\frac{l_C(\lambda)l_C(\frac pq)}{\mathrm{Area}(\partial C)}\le\frac{36(2g(K)-1)}{3.35},$$
where we use Cao and Meyerhoof's result \cite{CM} that $\mathrm{Area}(\partial C)\ge 3.35$.
\end{proof}

\begin{cor}\label{cor:Hyp}
If $S^3_{p/q}(K)\cong S^3_{p/q}(T_{r,s})$ holds for a knot $K$ and $|\frac pq|>\frac{30(r^2-1)(s^2-1)}{67}$, then $K$ is not hyperbolic.
\end{cor}
\begin{proof}
Since $r>s>1$, we have $\frac pq>\frac{30(r^2-1)(s^2-1)}{67}>(r-1)(s-1)>rs-r-s$. By Proposition~\ref{prop:GenusBound} we have $2g(K)-1\le\frac{(r^2-1)(s^2-1)}{24}$. If $K$ is hyperbolic, Lemma~\ref{lem:pBound} implies that
$$|\frac pq|\le|p|\le\frac{36}{3.35}\cdot\frac{(r^2-1)(s^2-1)}{24}=\frac{30(r^2-1)(s^2-1)}{67},$$
a contradiction.
\end{proof}

\begin{prop}\label{prop:Tmn}
Suppose that for a general torus knot $T_{m,n}\subset S^3$, we have
 $S^3_{p/q}(T_{m,n})\cong S^3_{p/q}(T_{r,s})$ for a slope $\frac pq\notin\{rs\pm1,rs\pm\frac12\}$.
 Then $T_{m,n}=T_{r,s}$.
\end{prop}
\begin{proof}
 By (\ref{eq:Delta''}) we have
\begin{equation}\label{eq:rsmn}
(r^2-1)(s^2-1)=(m^2-1)(n^2-1).
\end{equation}

If the  manifold $S^3_{p/q}(T_{m,n})\cong S^3_{p/q}(T_{r,s})$ is reducible,
then by \cite{Moser} the slope $p/q$ is $rs=mn$. Using (\ref{eq:rsmn}) we easily see that $T_{m,n}=T_{r,s}$.

If the  manifold $S^3_{p/q}(T_{m,n})\cong S^3_{p/q}(T_{r,s})$ is a lens space, it follows from \cite{Moser} that the slope $\frac pq$ satisfies
\begin{equation}\label{eq:Dist}
\Delta(\frac pq,rs)=|p-rsq|=\Delta(\frac pq,mn)=|p-mnq|=1
\end{equation}
From  (\ref{eq:Dist}), we have $rsq=mnq$ or $rsq=mnq\pm2$. If $rs=mn$, then using (\ref{eq:rsmn}) we get $T_{m,n}=T_{r,s}$. If $rsq=mnq\pm2$, then $|q|=1,2$ and $p=rs\pm\frac1q$.

So by \cite{Moser}, we may now assume that the  manifold $S^3_{p/q}(T_{m,n})\cong S^3_{p/q}(T_{r,s})$
is a Seifert fibred space whose base orbifold is $S^2$ with cone points, and the orders of the three
cone points  are
$$\{r,s,|p-rsq|\}=\{|m|,|n|,|p-mnq|\},$$
so $\{r,s\}\cap\{|m|,|n|\}\ne\emptyset$.
Without loss of generality, we may assume $r=|m|$, then $s=|n|$ by (\ref{eq:rsmn}). So $|p-rsq|=|p-mnq|$. If $T_{m,n}\ne T_{r,s}$, then $mn=-rs$ and so  we must have $p=0$. But clearly $S^3_0(T_{r,s})\ncong S^3_0(T_{r,-s})$. Hence  $T_{m,n}$ must be $T_{r,s}$.
\end{proof}

\begin{prop}\label{prop:Sat}
If $S^3_{p/q}(K)\cong S^3_{p/q}(T_{r,s})$ holds for a knot $K$
 and $$\frac pq>rs+\frac37\max\{r,s\}  \text{ with either } |p|>\frac{30(r^2-1)(s^2-1)}{67} \text{ or } |q|\ge3,$$ then $K$ is not a satellite knot.
\end{prop}
\begin{proof}
We may assume $p,q>0$.
As $\frac pq>rs+1$, $S^3_{p/q}(T_{r,s})$ is a Seifert fibred space  whose base orbifold is $S^2$ with three cone points of orders $r,s,  p-qrs$.

If $K$ is a satellite knot, let $R\subset S^3-K$ be the ``innermost'' essential torus, then $R$ bounds a solid torus $V\supset K$ in $S^3$. Let $K'$ be the core of $V$. The ``innermost'' condition means that $K'$ is not satellite, so it is either hyperbolic or a torus knot.
Note that $S^3_{p/q}(T_{r,s})$ is irreducible and does not contain incompressible torus.
So by \cite{GSurg} we know that $V_{p/q}(K)$ (the $p/q$-surgery on $V$ along $K$)
is a solid torus, $K$ is a braid in $V$, and  $S^3_{p/q}(K)\cong S^3_{p/(qw^2)}(K')$, where $w>1$ is the winding number of $K$ in $V$ and $\gcd(p,qw^2)=1$.

If $K'=T_{m,n}$ is a torus knot, since $K'$ admits a positive L-space surgery, we can choose $m,n>0$. Then
$S^3_{p/(qw^2)}(K')$ is a Seifert fibred space  whose base orbifold is $S^2$ with three cone points of orders $m,n,|p-qmnw^2|$. Now that $\{r,s,  p-qrs \}=\{m,n, |p-qmnw^2|\}$ and $w>1$. If $\{r,s\}=\{m,n\}$, then the $S^3_{p/q}(T_{r,s})\cong S^3_{p/(qw^2)}(T_{r,s})$, which is not possible.
So we may assume $m=r$, $n=p-qrs$ and $|p-qmnw^2|=s$. We get
$$p-qr(p-qrs)w^2=\pm s,$$
hence $$p=\frac{sq^2r^2w^2\mp s}{qrw^2-1}\le s(qr+\frac{qr+1}{4qr-1})\le sqr+\frac{3}7s,$$
where we use the fact that $r\ge2,w\ge2$. This contradicts the assumption that $\frac pq>rs+\frac37\max\{r,s\}$.

If $K'$ is hyperbolic, then $S^3_{p/(qw^2)}(K')$ is a non-hyperbolic  surgery.
Note that $g(K')\leq g(K)$. The argument in Corollary~\ref{cor:Hyp} implies that $p\le \frac{30(r^2-1)(s^2-1)}{67}$. Hence we should have $q\ge3$, then $qw^2\ge12$. It follows from \cite{LM} that $S^3_{p/(qw^2)}(K')$ is hyperbolic, a contradiction.
\end{proof}

\begin{proof}[Proof of Theorem~\ref{thm:TorusKnot}]Suppose that for a knot $K\subset S^3$
we have $S^3_{p/q}(K)=S^3_{p/q}(T_{r,s})$ for some slope $\frac pq>\frac{30(r^2-1)(s^2-1)}{67}$.
Note that a knot $K\subset S^3$ is either a hyperbolic knot, or a torus knot, or a satellite knot.
 Corollary~\ref{cor:Hyp} implies that $K$ is not hyperbolic. By \cite{OSz3141}, we may assume $T_{r,s}\ne T_{3,2}$, so $(r-1)(s-1)\ge4$. So $\frac pq>\frac{30(r^2-1)(s^2-1)}{67}>(r+1)(s+1)>rs+\frac37\max\{r,s\}$.
Proposition~\ref{prop:Sat} then implies that $K$ is not satellite. So we can apply Proposition~\ref{prop:Tmn}
to conclude that $K=T_{r,s}$.
\end{proof}

\section{Preliminaries on Heegaard Floer homology}\label{sect:Pre}

The rest of this paper is devoted to the special case of $T_{5,2}$. In order to study this case, we need to understand the Heegaard Floer homology in more details. In this section, we will give the necessary background on Heegaard Floer homology.

\subsection{Heegaard Floer homology and correction terms}

 Heegaard Floer homology, introduced by
  Ozsv\'ath and Szab\'o \cite{OSzAnn1}, is an invariant  for closed oriented $3$--manifolds $Y$ equipped
  with Spin$^c$ structures $\mathfrak s$,
  taking the form of a collection of related homology groups as  $\widehat{HF}(Y,\mathfrak s)$, $HF^{\pm}(Y,\mathfrak s)$, and $HF^\infty(Y,\mathfrak s)$.

\begin{rem}
For simplicity, throughout this paper we will use $\mathbb F=\mathbb Z/2\mathbb Z$ coefficients for
Heegaard Floer homology.
\end{rem}

There is an absolute Maslov $\mathbb Z/2\mathbb Z$--grading on the Heegaard Floer homology groups. When $\mathfrak s$ is torsion, there is an absolute Maslov $\mathbb Q$--grading on $HF^+(Y,\mathfrak s)$.
Let $J\co \spin(Y)\to\spin(Y)$ be the conjugation on $\spin(Y)$, then
\begin{equation}\label{eq:HFconj}
HF^+(Y,\mathfrak s)\cong HF^+(Y,J\mathfrak s)
\end{equation}
as ($\mathbb Z/2\mathbb Z$ or $\mathbb Q$) graded groups.

There is a $U$--action on $HF^+$, and the isomorphism (\ref{eq:HFconj}) respects the $U$--action.

For a rational homology three-sphere $Y$ with Spin$^c$ structure $\mathfrak s$, $HF^+(Y,\mathfrak s)$ can be decomposed as the direct sum of two groups: the first group is the image of $HF^\infty(Y,\mathfrak s)$ in $HF^+(Y,\mathfrak s)$,
which is isomorphic to $\mathcal T^+=\mathbb F[U,U^{-1}]/U\mathbb F[U]$, supported in the even $\mathbb Z/2\mathbb Z$--grading, and its minimal absolute  $\mathbb{Q}$--grading is an invariant of $(Y,\mathfrak s)$, denoted by $d(Y,\mathfrak s)$, the {\it correction term} \cite{OSzAbGr}; the second group is the quotient modulo the above image and is denoted by $HF_{\mathrm{red}}(Y,\mathfrak s)$.  Altogether, we have $$HF^+(Y,\mathfrak s)=\mathcal{T}^+\oplus HF_{\mathrm{red}}(Y,\mathfrak s).$$

The correction term satisfies
\begin{equation}\label{eq:CorrSymm}
d(Y,\mathfrak s)=d(Y,J\mathfrak s),\qquad d(-Y,\mathfrak s)=-d(Y,\mathfrak s).
\end{equation}

Let $L(p,q)$ be the lens space obtained by $\frac{p}q$--surgery on the unknot.
The correction terms for lens spaces can be computed inductively as follows:
\begin{eqnarray*}
d(S^3,0)&=&0,\\
d(-L(p,q),i)&=&\frac14-\frac{(2i+1-p-q)^2}{4pq}-d(-L(q,r),j),
\end{eqnarray*}
where $0\le i<p+q$,
$r$ and $j$ are the reductions modulo $q$ of $p$ and $i$,  respectively.

\subsection{The knot Floer complex}

Given a null-homologous knot $K\subset Y$, Ozsv\'ath--Szab\'o \cite{OSzKnot} and Rasmussen \cite{Ras} defined the knot Floer homology.
For knots in $S^3$, the knot Floer homology is a finitely generated bigraded group
$$\widehat{HFK}(K)=\bigoplus_{d,s\in\mathbb Z}\widehat{HFK}_d(K,s),$$
where $d$ is the Maslov grading and $s$ is the Alexander grading.
The Euler characteristic of the knot Floer homology is the Alexander polynomial. More precisely,
suppose
\begin{equation}\label{eq:Alexander}
\Delta_K(T)=\sum_{s\in\mathbb Z}a_sT^s
\end{equation}
is the normalized Alexander polynomial of $K$,
then $$a_s=\sum_{d\in\mathbb Z}(-1)^d\dim\widehat{HFK}_d(K,s).$$

Knot Floer homology is closely related to the topology of knots. It detects the Seifert genus of a knot \cite{OSzGenus}, and it determines whether the knot is fibred \cite{Gh,NiFibred}.

More information is contained in the knot Floer chain complex $$C=CFK^{\infty}(Y,K)=\bigoplus_{i,j\in\mathbb Z}C\{(i,j)\}.$$ The differential $\partial\co C\to C$  satisfies that
$$\partial^2=0,\quad\partial C\{(i_0,j_0)\}\subset C\{i\le i_0,j\le j_0\}.$$
Moreover, $H_*(C\{(i,j)\})\cong\widehat{HFK}_{*-2i}(Y,K,j-i)$, and there is a natural chain complex isomorphism $U\co C\{(i,j)\}\to C\{(i-1,j-1)\}$ which decreases the Maslov grading by $2$.
By \cite{Ras}, we can always assume
$$C\{(i,j)\}\cong\widehat{HFK}(Y,K,j-i).$$

There are quotient chain complexes
$$A^+_k=C\{i\ge0 \text{ or }j\ge k\},\quad k\in\mathbb Z$$
and $B^+=C\{i\ge0\}\cong CF^+(Y)$.
They satisfy that $A^+_k\cong A^+_{-k}$. As in \cite{OSzIntSurg}, there are chain maps
$$v_k,h_k\co A^+_k\to B^+.$$
Here $v_k$ is the natural vertical projection, and $h_k$ is more or less a horizontal projection. We often abuse the notation by calling the homology groups and induced maps on homology $A^+_k,B^+,v_k,h_k$.
Following \cite{NiWu}, let $A_k^T=U^nA^+_k$ for $n\gg0$, then $A_k^T\cong\mathcal T^+$. Let $A_{\mathrm{red},k}=A_k^+/A_k^T$.

When $Y=S^3$, $B^+\cong \mathcal T^+$.
The homogeneous map $v_k\co A_k^T\to B^+$ is $U$--equivariant, so it is equal to $U^{V_k}$ for some nonnegative integer $V_k$. Similarly, $h_k$ is equal to $U^{H_k}$ for some nonnegative integer $H_k$.
The numbers $V_k,H_k$ satisfy that
\begin{equation}\label{eq:VH}
V_k=H_{-k},\quad V_k\ge V_{k+1}\ge V_k-1.
\end{equation}

\begin{conv}\label{conv:Gr}
The groups $A^+_k$ will be relatively $\mathbb Z$--graded groups. For our convenience, we choose an absolute grading on $A^+_{k}$, such that $\mathbf 1\in A^T_k$ has grading $0$.
\end{conv}

For any positive integer $d$, define
$$\mathcal T_{d}=
\langle\mathbf 1,U^{-1},\dots,U^{1-d}\rangle\subset \mathcal T^+,
$$
and define
$$\mathcal T_0=0\subset \mathcal T^+.$$
Suppose $K\subset S^3$ has Alexander polynomial (\ref{eq:Alexander}), let
$$t_i(K)=\sum_{j=1}^{\infty}ja_{i+j},\quad i=0,1,2,\dots$$
Then the coefficients $a_s$ can be recovered by the formula
$$a_s=t_{s-1}-2t_s+t_{s+1}, \quad s=1,2,\dots$$

\begin{lem}\label{lem:kerV}
Suppose $K\subset S^3$.

(1) For any $k\ge0$, we have
$$\ker v_k\cong \mathcal T_{V_k}\oplus A_{\mathrm{red}, k}$$
and
$$\chi(\ker v_k)=t_k(K).$$
(2) Suppose $g=g(K)$ is the Seifert genus of $K$, then $\widehat{HFK}(K,g)\cong\ker v_{g-1}$ and
\begin{equation}\label{eq:Genus}
g-1=\max\{k|\ker v_k\ne0\}.
\end{equation}
\end{lem}
\begin{proof}
From the definition of $V_k$ and $A_{\mathrm{red}, k}$, it is clear that $\ker v_k\cong \mathcal T_{V_k}\oplus A_{\mathrm{red}, k}$.
The short exact sequence of chain complexes
$$
\begin{CD}
0@>>>C\{i<0,j\ge k\}@>>>A^+_k@>v_k>>B^+@>>>0
\end{CD}
$$
induces an exact triangle between the homology groups. Since $v_k$ is always surjective on the homology level, its kernel is the homology of $C\{i<0,j\ge k\}$, whose Euler characteristic is $t_k$.
In particular, $$\ker v_{g-1}\cong C\{(-1,g-1)\}\cong \widehat{HFK}(K,g)\ne0$$ and $\ker v_k=0$ when $k\ge g$.
\end{proof}

\subsection{The rational surgery formula}

The basic philosophy of knot Floer homology is, if we know all the information about it, then we can compute the Heegaard Floer homology of all the surgeries on $K$. Let us give more details about the surgery formula below, following \cite{OSzRatSurg}.

Let $$\mathbb A_i^+=\bigoplus_{s\in\mathbb Z}(s,A^+_{\lfloor\frac{i+ps}q\rfloor}(K)),\mathbb B_{i}^+=\bigoplus_{s\in\mathbb Z}(s,B^+).$$
Define maps
$$v_{\lfloor\frac{i+ps}q\rfloor}\co (s,A^+_{\lfloor\frac{i+ps}q\rfloor}(K))\to (s,B^+),\quad h_{\lfloor\frac{i+ps}q\rfloor}\co (s,A^+_{\lfloor\frac{i+ps}q\rfloor}(K))\to (s+1,B^+).$$
Adding these up, we get a chain map
$$D_{i,p/q}^+\co\mathbb A_i^+\to \mathbb B_i^+,$$
with
$$D_{i,p/q}^+\{(s,a_s)\}_{s\in\mathbb Z}=\{(s,b_s)\}_{s\in\mathbb Z},$$
where
$$b_s=v_{\lfloor\frac{i+ps}q\rfloor}(a_s)+h_{\lfloor\frac{i+p(s-1)}q\rfloor}(a_{s-1}).$$

In \cite{OSzRatSurg}, when $Y$ is an integer homology sphere there is an affine identification of $\mathrm{Spin}^c(Y_{p/q}(K))$ with $\mathbb Z/p\mathbb Z$, such that Theorem~\ref{thm:MC} below holds for each $i$ between $0$ and $p-1$. This identification can be made explicit by the procedure in \cite[Section~4, Section~7]{OSzRatSurg}. For our purpose in this paper, we only need to know that such an identification exists. We will use this identification throughout this paper.

We first recall the rational surgery formula \cite[Theorem~1.1]{OSzRatSurg}.

\begin{thm}[Ozsv\'ath--Szab\'o]\label{thm:MC}
Suppose that  $Y$ is an integer homology sphere, $K\subset Y$ is a knot. Let $\mathbb X^+_{i,p/q}$ be the mapping cone of $D_{i,p/q}^+$, then there is a graded isomorphism of groups
$$H_*(\mathbb X^+_{i,p/q})\cong HF^+(Y_{p/q}(K),i).$$
\end{thm}

The absolute $\mathbb Q$--grading on $\mathbb X^+_{i,p/q}$ can be determined as follows. We first require $D_{i,p/q}^+$ drops the absolute $\mathbb Q$--grading by $1$, then an appropriate absolute $\mathbb Q$--grading on $\mathbb B_i^+$ will determine the absolute grading of the mapping cone. We choose the absolute $\mathbb Q$--grading on $(s,B^+)$ such that the corresponding absolute grading on the mapping cone gives us the right grading on $HF^+(Y_{p/q}(\mathrm{unknot}),i)=HF^+(Y\#L(p,q),i)$. The absolute $\mathbb Z/2\mathbb Z$--grading can also be determined in this way.

More concretely, when $Y=S^3$ and $p,q>0$, the computation of $HF^+(S^3_{p/q}(\mathrm{unknot}),i)$ using Theorem~\ref{thm:MC} (see \cite{NiWu}) shows that the absolute $\mathbb Q$--grading on $\mathbb B_i^+$ is determined as follows.

Let $$
s_i=\left\{
\begin{array}{ll}
0, &\text{if }V_{\lfloor\frac{i}q\rfloor}\ge H_{\lfloor\frac{i+p(-1)}q\rfloor},\\
-1, &\text{if }V_{\lfloor\frac{i}q\rfloor}< H_{\lfloor\frac{i+p(-1)}q\rfloor}.
\end{array}
\right.
$$
Then
\begin{eqnarray}\label{eq:GrB}
\mathrm{gr}(s_i,\mathbf 1)&=&d(L(p,q),i)-1,\nonumber\\
\mathrm{gr}(s+1,\mathbf 1)&=&\mathrm{gr}(s,\mathbf 1)+2\lfloor\frac{i+ps}q\rfloor,\quad \text{for any }s\in\mathbb Z.\label{eq:GrOnB}
\end{eqnarray}

When $Y=S^3$, we can use the above theorem to compute $HF^+(S^3_{p/q}(K),i)\cong\mathcal T^+\oplus HF_{\mathrm{red}}(S^3_{p/q}(K),i)$. For the part that is isomorphic to $\mathcal T^+$, we only need to know its absolute $\mathbb Q$--grading, which is encoded in the correction term.
We recall the following proposition from \cite{NiWu}.

\begin{prop}\label{prop:Corr}
Suppose $K\subset S^3$, and $p,q>0$ are relatively prime integers. Then for any $0\le i\le p-1$ we have
$$d(S^3_{p/q}(K),i)=d(L(p,q),i)-2\max\{V_{\lfloor\frac{i}q\rfloor},H_{\lfloor\frac{i-p}q\rfloor}\}.$$
\end{prop}

Using (\ref{eq:VH}) and the fact that $-\lfloor\frac{i-p}q\rfloor=\lfloor\frac{p+q-1-i}q\rfloor$, we have $H_{\lfloor\frac{i-p}q\rfloor}=V_{\lfloor\frac{p+q-1-i}q\rfloor}$.
To simplify the notation, let
$$\delta_i=\delta_i(K)=\max\{V_{\lfloor\frac{i}q\rfloor},H_{\lfloor\frac{i-p}q\rfloor}\}=\max\{V_{\lfloor\frac{i}q\rfloor},V_{\lfloor\frac{p+q-1-i}q\rfloor}\}.$$
We note that $\delta_i$ can also be written as
\begin{equation}\label{eq:DeltaMin}
\delta_i=\min\left\{V_{\lfloor\frac{i+ps_i}q\rfloor},H_{\lfloor\frac{i+ps_i}q\rfloor}\right\}.
\end{equation}

Let $D_{i,p/q}^T$ be the restriction of $D_{i,p/q}^+$ on
$$\mathbb A_i^T=\bigoplus_{s\in\mathbb Z}(s,A^T_{\lfloor\frac{i+ps}q\rfloor}(K)).$$
When $Y=S^3$ and $p/q>0$, the map $D_{i,p/q}^T\co \mathbb A_i^T\to \mathbb B_i$ is onto \cite[Lemma~2.9]{NiWu}, hence $D_{i,p/q}^+\co \mathbb A_i^+\to \mathbb B_i$ is surjective.
So $HF^+(S^3_{p/q}(K),i)$ is isomorphic to $\ker D_{i,p/q}^+$.

The reduced part $HF_{\mathrm{red}}(S^3_{p/q}(K),i)$ comes in two parts. One part is the contribution of
$$\mathbb A_{\mathrm{red},i}=\bigoplus_{s\in\mathbb Z}(s,A_{\mathrm{red},\lfloor\frac{i+ps}q\rfloor}(K)).$$
The other part is a subgroup of  $\ker D_{i,p/q}^T$, which can be computed from the integers $V_k,H_k$.

\begin{prop}\label{prop:CompRed}Suppose that $K$ is a knot in $S^3$, $p,q>0$ are relatively prime integers.
Under the identification in Theorem~\ref{thm:MC}, the group
\begin{equation}\label{eq:RedGroup}
\mathbb A_{\mathrm{red},i}\oplus\left(\bigoplus_{s\in\mathbb Z}(s,\mathcal T_{\min\{V_{\lfloor \frac{i+ps}q\rfloor},H_{\lfloor \frac{i+ps}q\rfloor}\}})\right)/(s_i,\mathcal T_{\delta_i})
\end{equation}
is identified with
$HF_{\mathrm{red}}(S^3_{p/q}(K),i)$. Here it follows from (\ref{eq:DeltaMin}) that $(s_i,\mathcal T_{\delta_i})$ is a summand in the direct sum in (\ref{eq:RedGroup}).
\end{prop}
\begin{proof}
Since the sequence $V_k=H_{-k}$ ($k\in\mathbb Z$) is nonincreasing, we have
\begin{equation}\label{eq:si}
\begin{array}{ll}
H_{\lfloor \frac{i+p(s-1)}q\rfloor}\ge H_0=V_0\ge V_{\lfloor \frac{i+ps}q\rfloor} &\text{if }s>0,\\
H_{\lfloor \frac{i+p(s-1)}q\rfloor}\le H_0=V_0\le V_{\lfloor \frac{i+ps}q\rfloor} &\text{if }s<0.
\end{array}
\end{equation}

Given $\xi\in\mathcal T^+$, define
$$\rho(\xi)=\{(s,\xi_s)\}_{s\in\mathbb Z}$$
as follows. If
$$
V_{\lfloor\frac{i}q\rfloor}\ge H_{\lfloor\frac{i+p(-1)}q\rfloor},
$$
let
$$\xi_{-1}=U^{V_{\lfloor\frac{i}q\rfloor}-H_{\lfloor\frac{i+p(-1)}q\rfloor}}\xi,\quad \xi_{0}=\xi;$$
if
$$
V_{\lfloor\frac{i}q\rfloor}< H_{\lfloor\frac{i+p(-1)}q\rfloor},
$$
let
$$\quad\xi_{-1}=\xi,\quad \xi_{0}=U^{H_{\lfloor\frac{i+p(-1)}q\rfloor}-V_{\lfloor\frac{i}q\rfloor}}\xi.$$
For other $s$, using (\ref{eq:si}), let
$$
\xi_s=\bigg\{
\begin{array}{ll}
U^{H_{\lfloor\frac{i+p(s-1)}q\rfloor}-V_{\lfloor\frac{i+ps}q\rfloor}}\xi_{s-1},&\text{if }s>0,\\
U^{V_{\lfloor\frac{i+p(s+1)}q\rfloor}-H_{\lfloor\frac{i+ps}q\rfloor}}\xi_{s+1},&\text{if }s<-1.
\end{array}
$$

As in the proof of \cite[Proposition~1.6]{NiWu}, $\rho$ maps $\mathcal T^+$ injectively into $\ker D^T_{i,p/q}$ and $U\rho(\mathbf 1)=0$. Hence $\rho(\mathcal T^+)$ is the part of $\ker D^T_{i,p/q}$ which is isomorphic to $\mathcal T^+$.

Suppose $\eta=\{(s,\eta_s)\}_{s\in\mathbb Z}\in\ker D_{i,p/q}^T$. Let $\zeta=\eta-\rho(\eta_{s_i})\in\ker D_{i,p/q}^T$, then $\zeta_{s_i}=0$. Using (\ref{eq:si}), we can check $\zeta_s\in \mathcal T_{\min\{V_{\lfloor \frac{i+ps}q\rfloor},H_{\lfloor \frac{i+ps}q\rfloor}\}}$ for any $s\ne s_i$. So $\zeta$ is contained in the group (\ref{eq:RedGroup}). On the other hand, the group (\ref{eq:RedGroup}) is clearly in the kernel of $D_{i,p/q}^+$, so our conclusion holds.
\end{proof}

\begin{cor}\label{cor:RankRed}
Suppose $K\subset S^3$, and $p,q>0$ are relatively prime integers. Then
the rank of $HF_{\mathrm{red}}(S^3_{p/q}(K))$ is equal to
$$
q\big(\dim A_{\mathrm{red},0}+V_0+2\sum_{k=1}^{\infty}(\dim A_{\mathrm{red},k}+V_k)\big)-\sum_{i=0}^{p-1}\delta_i.
$$
\end{cor}
\begin{proof}
By Proposition~\ref{prop:CompRed}, the rank of $HF_{\mathrm{red}}(S^3_{p/q}(K))$ can be computed by
\begin{eqnarray*}
&&\sum_{i=0}^{p-1}\dim \mathbb A_{\mathrm{red},i}+\sum_{i=0}^{p-1}\sum_{s\in\mathbb Z}\min\{V_{\lfloor \frac{i+ps}q\rfloor},H_{\lfloor \frac{i+ps}q\rfloor}\}-\sum_{i=0}^{p-1}\delta_i\\
&=&\sum_{n\in\mathbb Z}\left(\dim A_{\mathrm{red},\lfloor\frac nq\rfloor}+\min\{V_{\lfloor \frac{n}q\rfloor},V_{\lfloor \frac{q-1-n}q\rfloor}\}\right)-\sum_{i=0}^{p-1}\delta_i\\
&=&q\left(\sum_{s\in\mathbb Z}\dim A_{\mathrm{red},s}+V_0+2\sum_{k=1}^{\infty}V_k\right)-\sum_{i=0}^{p-1}\delta_i.
\end{eqnarray*}
Since $\dim A_{\mathrm{red},s}(K)=\dim A_{\mathrm{red},-s}(K)$, we get our conclusion.
\end{proof}

\section{The genus bound in the $T_{5,2}$ case}\label{sect:GenusBound}

The purpose of this section is to show the following theorem.

\begin{thm}\label{thm:Fibered}
Suppose $K\subset S^3$, $\frac pq\in \mathbb Q$, and $S^3_{p/q}(K)\cong S^3_{p/q}(T_{5,2})$. Then one of the following two cases happens:

1) $K$ is a genus $(n+1)$ fibred knot for some $n\ge1$ with
\begin{equation}\label{eq:DeltaKn}
\Delta_K(T)=(T^{n+1}+T^{-(n+1)})-2(T^n+T^{-n})+(T^{n-1}+T^{1-n})+(T+T^{-1})-1.
\end{equation}

2) $K$ is a genus $1$ knot with $\Delta_K(T)=3T-5+3T^{-1}$.

Moreover,   if $$\frac pq\in\left\{\frac pq>1\right\}\cup \left\{
\frac pq<-6, |q|\ge2\right\},$$ then the number $n$ in the first case must be $1$.
\end{thm}

Suppose $S^3_0(K)=S^3_{0}(T_{5,2})$, then $K$ is a genus 2 fibred knot by Gabai \cite{G3}. Clearly, it must have the same Alexander polynomial as $T_{5,2}$. From now on we consider the case $\frac pq\ne0$.

\begin{figure}
\begin{center}
\begin{picture}(275,105)
\put(-20,40){\vector(1,0){125}} \put(102,32){$i$}

\put(40,-10){\vector(0,1){115}} \put(44,103){$j$}

\put(0,40){\circle*{4}}
\put(20,40){\circle*{4}}
\put(18,40){\textcolor{red}{\vector(-1,0){15}}}
\put(20,38){\textcolor{red}{\vector(0,-1){15}}}

\put(20,20){\circle*{4}}

\put(40,20){\circle*{4}}
\put(38,20){\textcolor{red}{\vector(-1,0){15}}}
\put(40,18){\textcolor{red}{\vector(0,-1){15}}}

\put(40,0){\circle*{4}}

\put(20,60){\circle*{4}}
\put(40,60){\circle*{4}}
\put(38,60){\textcolor{red}{\vector(-1,0){15}}}
\put(40,58){\textcolor{red}{\vector(0,-1){15}}}

\put(40,40){\circle*{4}}

\put(60,40){\circle*{4}}
\put(58,40){\textcolor{red}{\vector(-1,0){15}}}
\put(60,38){\textcolor{red}{\vector(0,-1){15}}}

\put(60,20){\circle*{4}}

\put(40,80){\circle*{4}}
\put(60,80){\circle*{4}}
\put(58,80){\textcolor{red}{\vector(-1,0){15}}}
\put(60,78){\textcolor{red}{\vector(0,-1){15}}}

\put(60,60){\circle*{4}}

\put(80,60){\circle*{4}}
\put(78,60){\textcolor{red}{\vector(-1,0){15}}}
\put(80,58){\textcolor{red}{\vector(0,-1){15}}}

\put(80,40){\circle*{4}}


\put(150,40){\vector(1,0){125}} \put(272,32){$i$}

\put(210,-10){\vector(0,1){115}} \put(214,103){$j$}

\put(170,40){\circle*{4}}
\put(170,38){\textcolor{red}{\vector(0,-1){15}}}

\put(170,20){\circle*{4}}

\put(190,20){\circle*{4}}
\put(188,20){\textcolor{red}{\vector(-1,0){15}}}
\put(190,18){\textcolor{red}{\vector(0,-1){15}}}

\put(190,0){\circle*{4}}

\put(210,0){\circle*{4}}
\put(208,0){\textcolor{red}{\vector(-1,0){15}}}

\put(190,60){\circle*{4}}
\put(190,58){\textcolor{red}{\vector(0,-1){15}}}

\put(190,40){\circle*{4}}

\put(210,40){\circle*{4}}
\put(208,40){\textcolor{red}{\vector(-1,0){15}}}
\put(210,38){\textcolor{red}{\vector(0,-1){15}}}

\put(210,20){\circle*{4}}

\put(230,20){\circle*{4}}
\put(228,20){\textcolor{red}{\vector(-1,0){15}}}

\put(210,80){\circle*{4}}
\put(210,78){\textcolor{red}{\vector(0,-1){15}}}

\put(210,60){\circle*{4}}

\put(230,60){\circle*{4}}
\put(228,60){\textcolor{red}{\vector(-1,0){15}}}
\put(230,58){\textcolor{red}{\vector(0,-1){15}}}

\put(230,40){\circle*{4}}

\put(250,40){\circle*{4}}
\put(248,40){\textcolor{red}{\vector(-1,0){15}}}

\end{picture}

\caption{\label{fig:T52} The knot Floer complexes of $T_{5,2}$ and $T_{5,-2}$, where a black dot stands for a copy of $\mathbb F$, and the arrows indicate the differential.}
\end{center}
\end{figure}
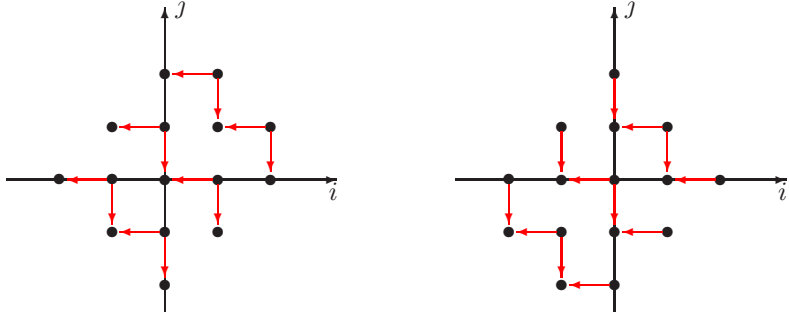

\begin{prop}\label{prop:ts}
Let $K_0$ be either $T_{5,2}$ or $T_{5,-2}$. Suppose $K\subset S^3$ is a knot with $S^3_{p/q}(K)\cong S^3_{p/q}(K_0)$ for $\frac pq>0$, then $$t_s(K)=\dim A_{\mathrm{red},s}(K)+V_s(K)=\dim\ker v_s$$ for any $s\ge0$.
\end{prop}
\begin{proof}
The knot Floer chain complexes of $T_{5,2}$ and $T_{5,-2}$ are illustrated in Figure~\ref{fig:T52}. It is easy to check that all $A_s(K_0)$ on homology level are supported in the even $\mathbb Z/2\mathbb Z$--grading. So $HF_{\mathrm{red}}(S^3_{p/q}(K_0))$ is supported in the even grading by Proposition~\ref{prop:CompRed}. Now it follows from Proposition~\ref{prop:CompRed} that $\ker v_s(K)$ is supported in the even grading for all $s\ge0$. By Lemma~\ref{lem:kerV}, $t_s(K)=\chi(\ker v_s)=\dim\ker v_s=\dim A_{\mathrm{red},s}(K)+V_s(K)$.
\end{proof}

\begin{prop}\label{prop:TwoCase}
Conditions are as in Proposition~\ref{prop:ts}. Then either $K$ is a genus $(n+1)$ fibred knot for some $n\ge1$, and $\Delta_K(T)$ is given by (\ref{eq:DeltaKn}), or $K$ is a genus $1$ knot with $\Delta_K(T)=3T-5+3T^{-1}$.
\end{prop}
\begin{proof}
By Proposition~\ref{prop:Corr},
\begin{equation}\label{eq:SumDelta}
\sum_{i=0}^{p-1}\delta_i(K)=\sum_{i=0}^{p-1}\delta_i(K_0).
\end{equation}
By Corollary~\ref{cor:RankRed} and Proposition~\ref{prop:ts}, we also have
\begin{equation}\label{eq:Sum=3}
t_0(K)+2\sum_{k>0}t_k(K)=t_0(K_0)+2\sum_{k>0}t_k(K_0)=3,
\end{equation}
and $t_k(K)\ge0$,
so $t_0(K)$ is equal to either $1$ or $3$.

If $t_0(K)=1$, then by (\ref{eq:Sum=3}) there is exactly one $n>0$ such that $\dim\ker v_n=t_n=1$, and $\dim\ker v_k=t_k$ is $0$ for $k>0$ and $k\ne n$. By Lemma~\ref{lem:kerV}, $g(K)=n+1$ and
(\ref{eq:DeltaKn}) holds.
Lemma~\ref{lem:kerV} also implies that $\widehat{HFK}(K,n+1)\cong\ker v_n\cong\mathbb F$, so $K$ is fibred \cite{NiFibred}.

If $t_0(K)=3$, then $t_s(K)=0$ for all $s>0$. It follows from Lemma~\ref{lem:kerV} and Proposition~\ref{prop:ts} that $g(K)=1$. Clearly, in this case $\Delta_K(T)=3T-5+3T^{-1}$.
\end{proof}

This finishes the proof of the first part of Theorem~\ref{thm:Fibered}.

\subsection{The case when $K_0=T_{5,2}$ and $\frac pq>1$}\label{subsect:>1}

In this subsection we will consider the case when $S^3_{p/q}(K)\cong S^3_{p/q}(K_0)$, where $K_0=T_{5,2}$, $p>q>0$, and $t_0(K)=t_n(K)=1$.
We can compute $V_0(K_0)=V_1(K_0)=1$, and $V_k(K_0)=0$ when $k\ge2$. Using (\ref{eq:VH}) and (\ref{eq:SumDelta}), we see that $V_0(K)>0$. Since $V_k\le t_k$, $V_0(K)=1\ge V_1(K)$ and $V_k(K)=0$ when $k\ge2$.

Now we have $V_k(K)\le V_k(K_0)$ for any $k\ge0$, so $\delta_i(K)\le\delta_i(K_0)$ for all $0\le i\le p-1$. In light of (\ref{eq:SumDelta}) we must have $\delta_i(K)=\delta_i(K_0)$ for all $0\le i\le p-1$.
If $V_1(K)=0$,
since $p>q$, we have $\delta_q(K)=\max\{V_1(K),V_{\lfloor\frac{p-1}q\rfloor}(K)\}=0$, but $\delta_q(K_0)=\max\{V_1(T_{5,2}),V_{\lfloor\frac{p-1}q\rfloor}(T_{5,2})\}=1$, a contradiction.
This shows that $V_1(K)=1$ hence $t_1(K)=1$. By Proposition~\ref{prop:TwoCase}, we have $g(K)=2$ and $\Delta_K(T)=(T^2+T^{-2})-(T+T^{-1})+1$.

\subsection{The case when $K_0=T_{5,-2}$ and $\frac pq>6$}\label{subsect:<-6}

In this subsection we will consider the case when $S^3_{p/q}(K)\cong S^3_{p/q}(K_0)$, where $K_0=T_{5,-2}$, $p>6q>0$, and $t_0(K)=t_n(K)=1$.

For the knot $K_0$, we have $V_k(K_0)=0$ when $k\ge0$. It then follows from (\ref{eq:SumDelta}) that $V_k(K)=0$ when $k\ge0$. By Proposition~\ref{prop:Corr}, we have
\begin{equation}\label{eq:NegCorr}
d(S^3_{p/q}(K),i)=d(S^3_{p/q}(K_0),i)=d(L(p,q),i), \quad i=0,1,\dots,p-1.
\end{equation}
By Lemma~\ref{lem:kerV}, we see that $A_{\mathrm{red},k}(K)\cong\mathbb F$ when $k=0,\pm n$, and $A_{\mathrm{red},k}(K)=0$ for all other $k$. Following Convention~\ref{conv:Gr}, $A_{\mathrm{red},k}(K)$ is absolutely graded. We assume $A_{\mathrm{red},0}(K)\cong\mathbb F_{(d_0)}$, $A_{\mathrm{red},\pm n}(K)\cong \mathbb F_{(d_n)}$, where $\mathbb F_{(d)}$ means a copy of $\mathbb F$ supported in grading $d$.

By Proposition~\ref{prop:CompRed}, $$HF_{\mathrm{red}}(S^3_{p/q}(K),i)\cong\mathbb A_{\mathrm{red},i}$$
for any $i\in\mathbb Z/p\mathbb Z$.

\begin{lem}\label{lem:AffMatch}
(1) Suppose that $\frac p2>q\ge2$, and let $N_q=\{0,1,\dots,q-1\}$. If $\phi\co\mathbb Z/p\mathbb Z\to\mathbb Z/p\mathbb Z$ is an affine isomorphism such that
$\phi(N_q)=N_q$, then either $\phi$ is the identity or $\phi(i)\equiv q-1-i\pmod p$.

(2) Suppose that $\frac p6>q\ge3$, and let $N^-_q=\{p-q,\dots,p-1\}$, $N^+_q=\{q,q+1,\dots,2q-1\}$. If $\phi\co\mathbb Z/p\mathbb Z\to\mathbb Z/p\mathbb Z$ is an affine isomorphism such that $$\phi(N_q)\subset N^-_q\cup N^+_q,$$
then either $\phi(N_q)=N^-_q$ or $\phi(N_q)=N^+_q$.
\end{lem}
\begin{proof}
(1) The result is obvious if $q=2$, now we assume $q>2$.
Suppose $\phi(i)\equiv ai+b\pmod p$.
For any $i\in N_q$, let $D_i=\{j-i|j\in N_q,\: j\ne i\}\subset \mathbb Z/p\mathbb Z$. Then the number $a$ satisfies that $a$ is contained in all but exactly one $D_i$. Noticing that $D_0\cap D_{q-1}=\emptyset$, so $a$ must be contained in $D_1\cap D_2\cap\cdots\cap D_{q-2}=\{-1,1\}$. If $a=1$, then $\phi=\mathrm{id}$. If $a=-1$, then $\phi(i)\equiv q-1-i\pmod p$.

(2) Suppose there exists such an affine isomorphism $\phi(i)\equiv ai+b\pmod p$ for some integers $a,b$. Without loss of generality, we can suppose $0<a\le \frac p2$.

Let $D^{\pm}\subset\mathbb Z/p\mathbb Z$ be the set of the differences between any element in $N^{\pm}_q$ and any element in $N^{\mp}_q$, and let $D\subset\mathbb Z/p\mathbb Z$ be the set of the differences between any two elements in $N^{\pm}_q$. Then
\begin{eqnarray*}
D^+&=&\{q+1,q+2,\dots,3q-1\},\\
D^-&=&\{p-3q+1,\dots,p-q-1\} ,\\
D&=&\{1,2,\dots,q-1\}\cup\{p-q+1,\dots,p-1\}.
\end{eqnarray*}

If $a\le q$, then $a\notin D^+\cup D^-$. So $\phi(N_q)$ must be either $N^-(q)$ or $N^+(q)$.

If $a>q$, we also have $a\le \frac p2< p-3q$, so $a\notin D^-\cup D$. This means that if $0\le i\le q-2$, then $\phi(i)\in N^-_q,\phi(i+1)\in N^+_q$, which forces $q\le2$, a contradiction.
\end{proof}

Given a rational homology sphere $Y$ and a Spin$^c$ structure $\mathfrak s$, let $\HFtildered(Y,\mathfrak s)$ be the group $HF_{\mathrm{red}}(Y,\mathfrak s)$, with the absolute grading shifted down by $d(Y,\mathfrak s)$.

\begin{lem}\label{lem:Conjugation}
Suppose $p>3q>0$, then the conjugation in $\spin(S^3_{p/q}(K))$ is given by $$J(i)\equiv q-1-i\pmod p.$$
\end{lem}
\begin{proof}
When $p>3q$, we can see that
\begin{equation}\label{eq:K0Surg}
\HFtildered(S^3_{p/q}(K_0),i)\cong\left\{
\begin{array}{ll}
\mathbb F_{(0)}, &\text{when }0\le i\le q-1,\\
\mathbb F_{(2)}, &\text{when }q\le i\le 2q-1\text{ or }p-q\le i\le p-1,\\
0, &\text{otherwise.}
\end{array}
\right.
\end{equation}
The conjugation $J$ on $\spin(S^3_{p/q}(K_0))$ is an affine involution on $\mathbb Z/p\mathbb Z$. By Equation (\ref{eq:HFconj}), we have $J(N_q)=N_q$. Since $p>2$, $J$ is not the identity map. If $q\ge2$, it follows from Lemma~\ref{lem:AffMatch} that $J(i)\equiv q-1-i\pmod p$ for $S^3_{p/q}(K_0)$. If $q=1$, by (\ref{eq:K0Surg}) we have $$J(0)=0,\quad J(1)=-1,$$ so we must have $J(i)\equiv -i\pmod p$.
Since the identification of $\spin(S^3_{p/q}(K))$ with $\mathbb Z/p\mathbb Z$ is purely homological, $J(i)\equiv q-1-i\pmod p$ should also be true for $S^3_{p/q}(K)$.
\end{proof}

In fact, the above lemma is true for any $p,q>0$ and $\frac pq$ surgery on any knot in homology spheres. This can be proved by examining the proof of \cite[Proposition~4.8]{OSzAbGr}.

\begin{lem}\label{lem:SameSpin}
Suppose that $\frac p6>q>0$, and $S^3_{p/q}(K)\cong S^3_{p/q}(K_0)$. Then
$$HF^+(S^3_{p/q}(K),i)\cong HF^+(S^3_{p/q}(K_0),i)$$
as $\mathbb Q$--graded groups
for any $i\in\mathbb Z/p\mathbb Z$, and the isomorphism respects the $U$--action.
\end{lem}
\begin{proof}
Since $S^3_{p/q}(K)\cong S^3_{p/q}(K_0)$, there is an affine isomorphism $\phi\co \mathbb Z/p\mathbb Z\to\mathbb Z/p\mathbb Z$ such that
$$HF^+(S^3_{p/q}(K),i)\cong HF^+(S^3_{p/q}(K_0),\phi(i))$$
as $\mathbb Q$--graded $\mathbb F[U]$--modules
for any $i\in\mathbb Z/p\mathbb Z$.

If $q=1$, by (\ref{eq:NegCorr}) we have
$$d(S^3_{p/q}(K),i)=d(S^3_{p/q}(K_0),i)=d(L(p,1),i)=-\frac14+\frac{(2i-p)^2}{4p}.$$
It is easy to check that the only affine isomorphisms $\phi\co \mathbb Z/p\mathbb Z\to\mathbb Z/p\mathbb Z$ satisfying that $d(S^3_{p/q}(K),i)=d(S^3_{p/q}(K_0),\phi(i))$
are $\phi(i)=i$ and $\phi(i)=J(i)$. Our conclusion holds by (\ref{eq:HFconj}).

If $q=2$,
$$d(L(p,2),i)=-\frac14+\frac{(2i-p-1)^2}{8p}+\left\{\begin{array}{ll}-\frac14,&i\text{ even,}\\ \frac14,&i\text{ odd.}\end{array}\right.$$
We can check that $d(L(p,2),i)$ attains its maximal value if and only if $i=0,1$.
If there is an affine isomorphism $\phi\co \mathbb Z/p\mathbb Z\to \mathbb Z/p\mathbb Z$ such that $d(L(p,2),i)=d(L(p,2),\phi(i))$, then $\phi$ must be either $\mathrm{id}$ or $J$.
We get our conclusion as in the last paragraph.

Now consider the case when $q\ge 3$. The group $A_{\mathrm{red},0}\cong\mathbb F_{(d_0)}$ contributes to each $\HFtildered(S^3_{p/q}(K),i)$ when $i\in N_q$. Comparing (\ref{eq:K0Surg}),
we must have
$$\HFtildered(S^3_{p/q}(K),i)\cong \mathbb F_{(d_0)},$$
and either $\phi(N_q)=N_q$ or $\phi(N_q)\subset N^+_q\cup N^-_q$.
If $\phi(N_q)=N_q$, then Lemma~\ref{lem:AffMatch} implies that $\phi=\mathrm{id}$ or $J$, hence our conclusion holds.
If $\phi(N_q)\subset N^+_q\cup N^-_q$, then Lemma~\ref{lem:AffMatch} implies that $\phi(N_q)=N_q^+$ or $\phi(N_q)=N_q^-$. However, as an isomorphism of Spin$^c$ structures induced by a homeomorphism, $\phi$ must commute with $J$. Since $J(N_q^{\pm})=N_q^{\mp}$ and $J(N_q)=N_q$, we get a contradiction.
\end{proof}

\begin{prop}\label{prop:>6}
Suppose that $\frac p6>q\ge2$, and $S^3_{p/q}(K)\cong S^3_{p/q}(K_0)$. Then $g(K)=2$ and $\Delta_K(T)=(T^2+T^{-2})-(T+T^{-1})+1$.
\end{prop}
\begin{proof}
The group $A_{\mathrm{red},k}$ is nontrivial only for $k=0,\pm n$, and the group $A_{\mathrm{red},0}$ contributes to $\mathbb A_{\mathrm{red},i}$ for each $i\in N_q$.
 The group $A_{\mathrm{red},n}$ contributes to $\mathbb A_{\mathrm{red},i}$ if and only if there exists an $s\in\mathbb Z$ such that $\lfloor\frac{i+ps}q\rfloor=n$ for some $s\in\mathbb Z$. For any $j\in\{0,1,\dots,q-1\}$, let $\mathfrak i(j)\in\{0,1,\dots,p-1\}$ satisfy that $\mathfrak i(j)+ps=nq+j$ for some $s\in\mathbb Z$. Then $\mathfrak i(0),\mathfrak i(1),\dots,\mathfrak i(q-1)$ are consecutive numbers in $\mathbb Z/p\mathbb Z$. By Lemma~\ref{lem:SameSpin} and (\ref{eq:K0Surg}), they should be either the numbers in $N^+_q$ or the numbers in $N^-_q$.

If these numbers are in $N^+_q$, then $\mathfrak i(j)=q+j$, and there exists a nonnegative integer $m$ such that $n=mp+1$. Then $\mathbb A_{\mathrm{red},\mathfrak i(j)}$ is supported in $(mq,A_{\mathrm{red},mp+1})$, and its absolute grading (see (\ref{eq:GrOnB})) is given by
$$d_n+1+\mathrm{gr}(mq,\mathbf 1)=d_n+d(L(p,q),q+j)+\sum_{s=0}^{mq-1}2\lfloor\frac{q+j+ps}q\rfloor,$$
where $(mq,\mathbf 1)$ is the lowest element in $(mq,B^+)\subset\mathbb B^+_{p/q,i}$.
Comparing with (\ref{eq:K0Surg}), we have
$$2=d_n+\sum_{s=0}^{mq-1}2\lfloor\frac{q+j+ps}q\rfloor,\text{ for any }j\in\{0,1,\dots,q-1\}.$$
This is impossible if $m\ge1,q\ge2$. In fact, there exists $\bar s\in\{0,1,\dots,q-1\}$ such that $q|(q+1+p\bar s)$, which implies that
$$\sum_{s=0}^{mq-1}2\lfloor\frac{q+ps}q\rfloor<\sum_{s=0}^{mq-1}2\lfloor\frac{q+1+ps}q\rfloor, \qquad \text{when }m\ge1,q\ge2.$$
So if $q\ge2$ we must have $m=0$, which implies $g(K)=n+1=mp+2=2$.

If these numbers are in $N^-_q$, then $\mathfrak i(j)=p-q+j$, and there exists a nonnegative integer $m$ such that $n=mp-1$. We can get a contradiction as before.
\end{proof}

\subsection{Finishing the proof of Theorem~\ref{thm:Fibered}}

If $\frac pq>0$, let $K_0=T_{5,2}$. If $\frac pq<0$, let $K_0=T_{5,-2}$, then the $-\frac pq$ surgery on the mirror of $K$ is homeomorphic to $S^3_{-p/q}(K_0)$ via an orientation preserving homeomorphism.
So we may always assume $\frac pq>0$ and $S^3_{p/q}(K)\cong S^3_{p/q}(K_0)$ for $K_0=T_{5,2}$ or $T_{5,-2}$.

If $t_0(K)=3$, the proof of Proposition~\ref{prop:TwoCase} shows that $g(K)=1$ and $\Delta_K(T)=3T-5+3T^{-1}$. If $t_0(K)=1$, then $g(K)=n+1$ and $\Delta_K(T)$ is given by (\ref{eq:DeltaKn}).

If $t_0(K)=1$, $K_0=T_{5,2}$ and $\frac pq>1$, then the result in Subsection~\ref{subsect:>1} shows that $K$ is a genus $2$ fibred knot with $\Delta_K(T)=(T^2+T^{-2})-(T+T^{-1})+1$.

If $t_0(K)=1$, $K_0=T_{5,-2}$ and $\frac pq>6$, then Proposition~\ref{prop:>6} implies $K$ is a genus $2$ fibred knot with $\Delta_K(T)=(T^2+T^{-2})-(T+T^{-1})+1$ unless $|q|=1$.

This finishes the proof of Theorem~\ref{thm:Fibered}.

\begin{rem}\label{addendum}
We have the following addendum to Theorem~\ref{thm:Fibered}:
\newline
(a) If $p$ is even, then case 2) of Theorem \ref{thm:Fibered} cannot happen and in case 1) of Theorem \ref{thm:Fibered}, the number $n$ must be odd.
\newline
(b) If $p$ is divisible by $3$, then case 2) cannot happen and in case 1), the number $n$ is not divisible by $3$.
\end{rem}

\begin{proof}
For a knot $K$ in $S^3$, let $M_K$ be its exterior and $\{\m,\l\}$ be the standard meridian-longitude basis on $\p M_K$.
For an integer $j>1$, let  $f_j:\widetilde M^j_K\ra M_K$ be the unique $j$-fold
free cyclic covering. On $\p \widetilde M_K^j$ we choose the basis $\{\tilde \m_j,\tilde \l_j\}$ such that
$f_j(\tilde \m_j)=\m^j$ and $f(\tilde \l_j)=\l$.
We also use $M_K(p/q)$ to denote $S^3_{p/q}(K)$, and use $\widetilde M^j_K(p/q)$ to denote the Dehn filling of $\widetilde M_K^j$
with slope $p/q$ with respect to the basis $\{\tilde \m_j, \tilde \l_j\}$.
Note that $\widetilde M_K^j(1/0)$ is the unique $j$-fold cyclic branched cover of $S^3$ branched over $K$.
By \cite[Theorem~8.21]{BurdeZieschang}, if no root of $\D_K(T)$ is a $j$-th root of unity, then  the order of the first homology of  $\widetilde M_K^j(1/0)$ is
 $$|H_1(\widetilde M_K^j(1/0))|=|\Pi_{i=1}^j\D(\xi_i)|,$$ where $\xi_1,...,\xi_j$ are the $j$ roots of
 $x^j=1$, in which case it follows from \cite[Proposition~8.19]{BurdeZieschang}
 that $$|H_1(\widetilde M_K^j(p/q))|=|\z/p\z\oplus H_1(\widetilde M_K^j(1/0))|
 =|p|\cdot|\Pi_{i=1}^j\D(\xi_i)|.$$

(a) Suppose $p$ is even. Let $\tilde p=p/2$. Then $\widetilde M_K^2(\tilde p/q)$ is the unique free double cover of $M_K(p/q)$.
If $\D_K(T)=3(T+T^{-1})-5$ or if $\D_K(T)=(T^{(n+1)}+T^{-(n+1)})-2(T^n+T^{-n})+(T^{(n-1)}+T^{-(n-1)})+(T+T^{-1})-1$ for $n$ even,
then  $|H_1(\widetilde M_K^2(\tilde p/q))|=|\tilde p||\D_K(-1)|=11|\tilde p|$. But
 $|H_1(\widetilde M_{T_{5,2}}^2(\tilde p/q))|=5|\tilde p|$.
Hence case 2) cannot happen and $n$ must be odd.

(b)
Suppose $p$ is divisible by $3$. The argument  is similar.
Let $\tilde  p=p/3$. Then $\widetilde M_K^3(\tilde p/q)$ is the unique free $3$-fold cyclic cover of $M_K(p/q)$. If  $\D_K(T)=3(T+T^{-1})-5$ or
$\D_K(T)=(T^{(n+1)}+T^{-(n+1)})-2(T^n+T^{-n})+(T^{(n-1)}+T^{-(n-1)})+(T+T^{-1})-1$ for $n$ divisible
by $3$, then
 $|H_1(\widetilde M_K^3(\tilde p/q))|=64|\tilde p|$. But  $|H_1(\widetilde M_{T_{5,2}}^3(\tilde p/q))|
=|\tilde p|$. Hence case 2) cannot happen and $n$ is not divisible by $3$.
\end{proof}

\section{Finishing the proof of Theorem~\ref{thm:T52}}\label{sect:T52}

\begin{prop}\label{prop:q9}
Suppose that $S^3_{p/q}(K)\cong S^3_{p/q}(T_{5,2})$  and $|q|\ge9$, then $K=T_{5,2}$.
\end{prop}
\begin{proof}
Since $|q|\ge9$, the result of \cite{LM} implies that $K$ is not hyperbolic.

If $K$ is a torus knot, then the computation of $\Delta_K(T)$ in Theorem~\ref{thm:Fibered} implies that $K=T_{5,2}$ or $T_{5,-2}$, and now it is easy to see that $K=T_{5,2}$.

If $K$ is a satellite knot, as in the proof of Proposition~\ref{prop:Sat}, we may assume the companion knot $K'$ is not a satellite knot. Since $S^3_{p/q}(K)\cong S^3_{p/q}(T_{5,2})$ does not contain any incompressible tori and $|q|\ge9>1$, it follows from \cite{GSurg,G1bridge}
that $K$ is a $(\pm a, b)$--cable of $K'$ for some integers $a,b$ with $a>0,b>1$, and $S^3_{p/q}(K)\cong S^3_{p/(qb^2)}(K')$. Again, by \cite{LM} $K'$ is not hyperbolic, so $K'$ is a torus knot.

Let us recall the formula for the Alexander polynomial of a satellite knot \cite[Proposition 8.23]{BurdeZieschang}. If $S$ is a satellite knot with pattern knot $P$ and companion knot $C$ and the winding number $w$ of $S$ in a regular neighborhood of $C$, then the Alexander polynomials of $S$, $C$ and $P$ satisfy
the following relation
\begin{equation}\label{polynomial relation}
\Delta_S(T)=\Delta_C(T^w)\cdot\Delta_P(T).
\end{equation}
In our present case, we see immediately that $\Delta_K(T)$ is monic since both the pattern knot and the companion knot of $K$ have monic Alexander polynomials. So $\Delta_K(T)$ is given by
(\ref{eq:DeltaKn}). Now if $a=1$, then the pattern knot is trivial, so $\Delta_K(T)=\Delta_{K'}(T^b)$. The right hand side of (\ref{eq:DeltaKn}) is not of this form, so this case does not happen.

Now we have $a>1$, and $K'=T_{\pm c, d}$ for $c,d>1$. The pattern knot is $T_{\pm a, b}$, whose Alexander polynomial has the form
$$T^{g_{a,b}}-T^{g_{a,b}-1}+\text{lower order terms},\quad \text{where } g_{a,b}=\frac{(a-1)(b-1)}2.$$
Similarly, $$\Delta_{K'}(T^b)=T^{bg_{c,d}}-T^{bg_{c,d}-b}+\text{lower order terms}.$$
Hence by (\ref{polynomial relation}) $$\Delta_K(T)=T^{bg_{c,d}+g_{a,b}}-T^{bg_{c,d}+g_{a,b}-1}+\text{lower order terms},$$ which could be equal to the right hand side of (\ref{eq:DeltaKn}) only when $n=1$. However, in this case the degrees of the highest terms of the two polynomials do not match, so this case does not happen.
\end{proof}

\begin{lem}\label{lem:SorT}
Suppose that $K$ is a satellite knot or a torus knot, $g(K)\le2$,  $K$ is fibred when $g(K)=2$,
 and $S^3_{p/q}(K)$ is homeomorphic to $S^3_{p/q}(T_{5,2})$ (not necessarily orientation preserving) for a nontrivial slope $p/q\ne0$. Then $K=T_{5,2}$.
\end{lem}

\begin{proof}
If $K$ is a torus knot, then from $g(K)\le 2$ we know that $K$ is one of $T_{5,\pm2}$ or $T_{3,\pm2}$.
Now it is easy to see that   $K=T_{5,2}$.

So suppose that  $K$ is a satellite knot, with companion knot $C$,   pattern knot $P$ and winding number  $w$
recalled as above. Since $S^3_{p/q}(K)$ does not contain any incompressible tori, it follows from \cite{GSurg} and \cite{S} that $w>1$ and $P$ is a $w$--braid in the solid torus.

By \cite[Proposition 2.10]{BurdeZieschang}, $g(K)\ge wg(C)+g(P)\ge w>1$.
So $K$ is fibred of genus $2$, $g(C)=1$, $g(P)=0$, $w=2$.

The companion knot $C$ must also be fibred (see, e.g. \cite{HMS}).
Therefore $C$ is either the trefoil knot or the figure-$8$ knot, and $K$ is a $(\pm1, 2)$ cable
on $C$.
It follows from  \cite[Lemma 7.2 and Corollary 7.3]{Go} that either
$p/q=\pm2$ and $S^3_{\pm2}(K)=S^3_{\pm1/2}(C)\# L(2,1)$
or  $p/q=(\pm2q+\varepsilon)/q$
and $S^3_{(\pm2q+\varepsilon)/q}(K)=S^3_{(\pm2q+\varepsilon)/(4q)}(C)$, where $\varepsilon$ is $1$ or $-1$.
It follows that  $C$ cannot be the figure-8 knot since all non-integral surgery on the
figure-8 knot yields hyperbolic manifolds.
So  $C$ is the trefoil knot.
As   $S^3_{\pm1/2}(C)$ can  never  be a lens space of order $5$,
we have $p/q=(\pm2q+\varepsilon)/q$
and $S^3_{(\pm2q+\varepsilon)/q}(K)=S^3_{(\pm2q+\varepsilon)/(4q)}(C)$.
Now $S^3_{(\pm2q+\varepsilon)/q}(T_{5,2})$ is Seifert fibred over $S^2(2,5, |\pm2q+\varepsilon-10q|)$
while  $S^3_{(\pm2q+\varepsilon)/(4q)}(C)$ is Seifert fibred over
$S^2(2,3, |\pm2q+\varepsilon-24q|)$ or $S^2(2,3, |\pm2q+\varepsilon+24q|)$
depending on $C$ is right-hand or left hand trefoil.
Hence $|\pm2q+\varepsilon-10q|=3$ and $|\pm2q+\varepsilon-24q|=5$ (or $|\pm2q+\varepsilon+24q|=5$),
which is not possible.
\end{proof}

\begin{cor}\label{cor:p>=33}
Suppose that $S^3_{p/q}(K)\cong S^3_{p/q}(T_{5,2})$  with $|p|\ge33$ and $g(K)\le2$, then $K=T_{5,2}$.
\end{cor}
\begin{proof}
By Lemma~\ref{lem:pBound}, $K$ is not a hyperbolic knot. Theorem~\ref{thm:Fibered} implies that $K$ is fibred if $g(K)=2$. Our conclusion follows from Lemma~\ref{lem:SorT}.
\end{proof}

\begin{lem}\label{lamination}
Suppose that $K\subset S^3$ is a hyperbolic fibred knot.
 Then $S^3_{p/q}(K)$ is hyperbolic if    $|q|\geq 3$ and $1\leq |p|\leq 2|q|-3$.
 \end{lem}

\begin{proof}
By \cite[Theorem 5.3]{GO}, there is an essential lamination in the complement of $K$ with a degenerate
 slope $\g_0$ such that $\g_0$ is either the trivial slope or an integer slope of $K$.
 Also by \cite[Theorem 8.8]{G2} (or \cite[Corollary 7.2]{Ro}), if $\g_0$ is an integer slope, then $|\g_0|\geq 2$.
 Furthermore by \cite[Theorem 2.5]{Wu} combined with the geometrization theorem of Perelman, $S^3_{p/q}(K)$ is hyperbolic if $\D(p/q, \g_0)>2$.
Hence if $\g_0=1/0$, then $S^3_{p/q}(K)$ is hyperbolic for $|q|\geq 3$.
So we may assume that $\g_0$ is an integer with $|\g_0|\geq 2$.
Now $\D(p/q, \g_0)=|p-q\g_0|\geq |q\g_0|-|p|\geq 3$ by our condition on $p/q$
and thus  $S^3_{p/q}(K)$ is hyperbolic.
\end{proof}

\begin{cor}\label{|p|<4}
Suppose that $S^3_{p/q}(K)\cong S^3_{p/q}(T_{5,2})$  with $|q|\geq 3$ and $2\leq |p|\leq 2|q|-3$, then $K=T_{5,2}$.
\end{cor}
\begin{proof}By Theorem \ref{thm:Fibered}, $K$ is either a fibred knot or a genus one
 knot. Note that $S^3_{p/q}(K)\cong S^3_{p/q}(T_{5,2})$ is Seifert fibred over $S^2$ with three singular fibers. So if $K$ is a genus one hyperbolic knot, then it follows from \cite[Theorem 1.5]{BGZ} that $|p|\leq 3$,
 and by Remark \ref{addendum}, $|p|\ne 2$ or $3$.
 Hence by our assumption on $p/q$, $K$ is not a genus one hyperbolic  knot.
 So by Lemma \ref{lem:SorT}, we may assume that $K$ is   not a genus one knot.
 Hence $K$ is a fibred knot.
Now the proof proceeds similarly to that of Proposition~\ref{prop:q9}, using  Lemma~\ref{lamination} instead of \cite{LM}. We only need to note that if $K$ is a fibred satellite knot, then any companion  knot of $K$
is also fibred.\end{proof}

\begin{proof}[Proof of Theorem~\ref{thm:T52}] Suppose that $S^3_{p/q}(K)\cong S^3_{p/q}(T_{5,2})$
for some slope $p/q$ in the set
\begin{center}$\displaystyle\left\{\frac pq>1, |p|\ge 33\right\}\cup\left\{
\frac pq<-6, |p|\ge33, |q|\ge2\right\}\cup \left\{\frac pq,  |q|\ge9\right\}
\cup\left\{\frac pq,   |q|\geq 3, 2\leq |p|\leq 2|q|-3\right\}$\\
$\displaystyle\cup\left\{9,10,11,\frac{19}{2},\frac{21}2,
\frac{28}{3}, \frac{29}{3}, \frac{31}{3}, \frac{32}{3} \right\}.$
\end{center}

If $\frac pq>1$ and $|p|\ge33$, by Theorem~\ref{thm:Fibered} we know $g(K)\le2$. Now by  Corollary~\ref{cor:p>=33} we have  $K=T_{5,2}$.

If $\frac pq<-6$, $|q|\ge2$ and $|p|\ge33$, the argument is as in the preceding case.

If $|q|\ge9$, then $K=T_{5,2}$ by Proposition~\ref{prop:q9}.

If  $|q|\geq 3$ and $2\leq |p|\leq 2|q|-3$, then  $K=T_{5,2}$ by Corollary \ref{|p|<4}.

If $\frac pq=9$ or $11$, then $S^3_{p/q}(K)\cong S^3_{p/q}(T_{5,2})$ is a lens space
of order $9$ or $11$ respectively.
By \cite[Theorem 1.6]{Baker}, $K$ is not a hyperbolic knot.
 As $K$ is either a genus two fibred knot or a genus one knot
 by Theorem~\ref{thm:Fibered}, we must have $K=T_{5,2}$ by  Lemma \ref{lem:SorT}.

If $\frac pq=10$, then $K=T_{5,2}$ as remarked in Section~\ref{Sect:Intr} (just before Theorem \ref{thm:TorusKnot}).

If $\frac pq=\frac{19}{2},\frac{21}2,
\frac{29}{3}$, or $\frac{31}{3}$, then $S^3_{p/q}(K)\cong S^3_{p/q}(T_{5,2})$ is a lens space.
By the cyclic surgery theorem of \cite{CGLS}, $K$ is not a hyperbolic knot.
Hence it follows from  Theorem~\ref{thm:Fibered} and Lemma \ref{lem:SorT} that $K=T_{5,2}$.

If $\frac pq=\frac{28}{3}$ or $\frac{32}3$,  then $S^3_{p/q}(K)\cong S^3_{p/q}(T_{5,2})$  is a spherical space form.
By \cite[Corollary 1.3]{BZ}, $K$ is not a hyperbolic knot.
Again  by Theorem~\ref{thm:Fibered} and Lemma \ref{lem:SorT} we have $K=T_{5,2}$.
\end{proof}

\end{document}